\documentclass[final]{siamltex1213}
\usepackage{a4,amsmath,amssymb,amscd,latexsym,amstext,amsxtra} 
\usepackage{cite}
  
\newtheorem{remark}{Remark}[section]
\newtheorem{algo}{Algorithm}[section]

\title{Projected Reflected Gradient Methods for Monotone Variational Inequalities} 

\author{Yu. Malitsky\footnotemark[2]}

\newcommand{\n}[1]{\left\|#1 \right\|} 
\renewcommand{\a}{\alpha}
\renewcommand{\b}{\beta}
\renewcommand{\c}{\gamma}
\renewcommand{\d}{\delta}

\newcommand{\la}{\lambda}

\newcommand{\e}{\varepsilon}
\renewcommand{\t}{\tau}

\newcommand{\x}{\bar x}
\newcommand{\R}{\mathbb R}

\renewcommand {\r}{\rightarrow}
\newcommand{\rr}{\rightharpoonup} 
\newcommand{\lr}[1]{\left\langle #1\right\rangle} 

\begin{document}
\maketitle
\slugger{mms}{xxxx}{xx}{x}{x--x}
\renewcommand{\thefootnote}{\fnsymbol{footnote}}
\footnotetext[2]{Department of Cybernetics, Taras Shevchenko National
  University of Kyiv,\\  64/13,  Volodymyrska Str.,  Kyiv, 01601,
  Ukraine (\email{y.malitsky@gmail.com).}}
 
\begin{abstract}
This paper is concerned with some new projection methods for solving variational inequality problems
with monotone and Lipschitz-continuous mapping in Hilbert space. First, we propose the projected
reflected gradient algorithm with a constant stepsize. It is similar to  the projected gradient
method, namely, the method requires  only one projection onto the  feasible set and only one value
of the mapping per iteration. This distinguishes our method from  most other projection-type methods
for variational inequalities  with monotone mapping. Also we prove that it has R-linear rate of
convergence under the strong monotonicity assumption.   The usual drawback of algorithms with
constant stepsize is the requirement to know the Lipschitz constant of the mapping. To avoid this, we
modify our first algorithm so that the algorithm needs at most  two projections per  iteration. In fact,  our computational experience shows that such cases with two projections are very rare. This scheme, at least theoretically, seems to be very effective.  All methods are shown to be globally convergent to a solution of the variational inequality. Preliminary results from numerical experiments are quite promising.
\end{abstract}

\begin{keywords}
variational inequality, projection method, monotone mapping, extragradient method
\end{keywords}

\begin{AMS}
 47J20, 90C25, 90C30, 90C52
\end{AMS}

\pagestyle{myheadings}
\thispagestyle{plain}
\markboth{YU. MALITSKY}{PROJECTED REFLECTED GRADIENT METHODS}

\section{Introduction}\label{intro}
We consider the classical variational inequality problem (VIP) which is to find a point $x^*\in C$ such that
\begin{equation}\label{vip}
\lr{F(x^*), x-x^*}\geq 0 \quad \forall x\in C,
\end{equation}
where $C$ is a closed convex set in Hilbert space $H$, $\lr{\cdot, \cdot}$ denotes the inner product in $H$, and $F\colon H\r H$ is  some mapping.
We assume that the following conditions hold
\begin{flushleft}
\begin{itemize}
\item[(C1)]\quad The solution set of \eqref{vip}, denoted by $S$, is nonempty.
\item[(C2)]\quad The mapping $F$ is monotone, i.e.,
$$\lr{F(x)-F(y),x-y}\geq 0\quad \forall x,y \in H.$$
\item[(C3)]\quad  The mapping $F$ is Lipschitz-continuous with constant $L>0$, i.e., there exists $L>0$ such that
$$\n{F(x)-F(y)}\leq L \n{x-y} \quad \forall x,y\in H.$$
\end{itemize}
\end{flushleft}

The variational inequality problem is one of the central problems in nonlinear analysis (see \cite{baiocchi:84,pang:03,zeidler3}).  Also monotone operators have turned out to be an important tool in the study of various problems arising in the domain of optimization, nonlinear analysis, differential equations and other related fields (see \cite{baucomb,zeidler2}). Therefore, numerical methods for VIP with monotone operator have been extensively studied in the literature, see \cite{pang:03,konnov07} and references therein. In this section we briefly consider the development of projection methods for monotone variational inequality that provide weak convergence to a solution of \eqref{vip}. 
 
The simplest iterative procedure is the well-known projected gradient method 
\begin{equation*}
  x_{n+1}=P_C (x_n- \la F(x_n)),
\end{equation*}
where $P_C$ denotes the metric projection onto the set $C$ and $\la$ is some positive number.  In order to converge, however, this method requires the restrictive assumption that $F$ be strongly (or inverse strongly) monotone. The extragradient
method proposed by Korpelevich and Antipin \cite{korpel:76,antipin}, according to the following formula, overcomes this difficulty
\begin{equation}\label{korpel}
\begin{cases}
     y_{n}=P_C(x_n-\la F(x_n))\\
     x_{n+1}=P_C(x_n-\la F(y_n)),
   \end{cases}
 \end{equation}
where $\la \in (0,\frac 1 L)$. The extragradient method has received a great deal of attention by many authors, who
improved it in various ways; see, e.g., \cite{khobotov,iusem:97,solodov:1999,reich:2011} and references
therein. We restrict our attention only to one extension of the extragradient
method. It was proposed in \cite{reich:2011} by Censor, Gibali and
Reich  
\begin{equation}\label{censor}
\begin{cases}  
y_{n} = P_C(x_n-\la F(x_n))\\
T_n=\{w \in H \mid \lr{x_n-\la F(x_n)-y_n,w-y_n}\leq 0\}\\
x_{n+1} = P_{T_n}(x_n-\la F(y_n)),
\end{cases}
\end{equation}
where $\la \in (0,\frac 1 L)$.
Since the second projection in \eqref{censor} can be found in a closed form, this method is more applicable when a projection
onto the closed convex set $C$ is a nontrivial problem. 

As an alternative to the extragradient method or its modifications  is  the following remarkable scheme proposed by
Tseng in \cite{tseng00}
\begin{equation}  \label{tseng}
\begin{cases}
     y_n = P_C(x_n-\la F(x_n))\\
     x_{n+1} = y_n +\la (F(x_n)-F(y_n)),
\end{cases}
\end{equation}
where $\la\in (0,\frac 1 L)$. Both algorithms  \eqref{censor} and \eqref{tseng} have the same complexity per iteration:  we need to compute one projection onto the set $C$ and two values of $F$.

Popov in his work \cite{popov:80} proposed an ingenious method, which is similar to the extragradient method, but uses on every iteration only one value of the mapping $F$.  Using the idea from \cite{reich:2011,lyashko11}, Malitsky and Semenov improved Popov algorithm. They presented in \cite{mal-sem} the following algorithm
\begin{equation}\label{mal_sem}
\begin{cases}
 T_n=\{w\in H\mid \lr{x_n-\la F(y_{n-1})-y_n,w-y_n}\leq 0\}\\
x_{n+1}=P_{T_n}(x_n-\la F(y_n))\\
y_{n+1}=P_C(x_{n+1}-\la F(y_n)),
\end{cases}
\end{equation} 
where $\la \in (0,\frac{1}{3L}]$. It is easy to see that this method needs only one projection onto the set $C$ (as in \eqref{censor} or \eqref{tseng}) and only one value of $F$ per iteration. The latter makes algorithm \eqref{mal_sem} very attractive for cases when a computation of operator $F$ is expensive. This  often happens, for example, in a huge-scale VIP or VIP that arises from optimal control.

In this work we propose the following scheme
\begin{equation}\label{our}
x_{n+1} = P_{C}(x_n-\la F(2x_n-x_{n-1})),
\end{equation}
where  $\la\in (0,\frac{\sqrt 2 -1}{L})$. Again we see that both algorithms~\eqref{mal_sem}
and~\eqref{our} have the same  computational complexity per iteration, but the latter has more
simple and elegant structure. Algorithm \eqref{our} reminds the projected gradient method, however,
a value of a gradient is taken in the point that is a reflection of $x_{n-1}$ in $x_n$. Preliminary 
results from numerical experiments, comparing algorithm \eqref{our} to others, are promising. Also
note that a simple structure of \eqref{our} allows us to obtain a nonstationary variant of
~\eqref{our} with variable stepsize in a very convenient way.

The paper is organized as follows. Section~\ref{prel}  presents some basic useful facts which we use throughout the paper. In Section~\ref{algorithm} we prove the convergence of the method~\eqref{our}. Under some more restrictive assumption we also establish its rate of convergence.  In Section~\ref{modif_algorithm} we present the nonstationary Algorithm~\ref{algModif1} which is more flexible, since it does not use the Lipschitz constant of the mapping $F$. This makes it much more convenient in practical applications than Algorithm~\ref{our}.   Section~\ref{num_results} contains the results of numerical experiments.

\section{Preliminaries}\label{prel}
The next three statements are classical. For the proof we refer the reader to \cite{baiocchi:84,baucomb}.

\begin{lemma}\label{proj}
  Let $M$ be nonempty closed convex set in $H$, $x \in H$. Then
  \begin{romannum}
  \item $\lr{P_Mx - x, y-P_Mx}\geq 0\quad \forall y\in M$;
  \item $\n{P_Mx-y}^2\leq \n{x-y}^2 - \n{x-P_Mx}^2 \quad \forall y\in M$.
  \end{romannum}
\end{lemma}
\ 

\begin{lemma}[Minty]\label{minty}
  Assume that $F:C\r H$ is a continuous and monotone mapping. Then $x^*$ is a solution of \eqref{vip} iff $x^*$ is a solution of the following problem
$$\text{find}\quad x\in C \quad \text{such that}\quad \lr{F(y),y-x}\geq 0\quad \forall y\in C.$$
\end{lemma}
\begin{remark}
  The solution set $S$ of variational inequality \eqref{vip} is closed and convex.
\end{remark}

As usual, the symbol  $\rr$ denotes  a weak convergence in $H$.
\begin{lemma}[Opial] \label{opial} Let $(x_n)$ be a sequence in $H$ such that $x_n\rr x$. Then for all $y\neq x$
$$\liminf_{n\to \infty}\n{x_n-x} <\liminf_{n\to \infty }\n{x_n-y}.$$
\end{lemma}

The proofs of the next four statements are omitted by their simplicity.

\begin{lemma}\label{sunny}
  Let $M$ be nonempty closed convex set in $H$, $x \in H$ and $\x = P_M x$. Then for all $\la>0 $ $$P_M(\x + \la (x-\x)) = \x.$$
\end{lemma}

\begin{lemma}\label{u_v}
  Let $u$, $v$ be in H. Then
$$\n{2u-v}^2 = 2\n{u}^2 - \n{v}^2 + 2\n{u-v}^2.$$  
\end{lemma}

\begin{lemma}\label{liminf}
  Let $(a_n)$ and $(b_n)$ be two real sequences. Then
$$\liminf_{n \to \infty} a_n + \liminf_{n \to \infty} b_n \leq \liminf_{n\to\infty} (a_n+b_n).$$
\end{lemma}

\begin{lemma}\label{lim_seq}
  Let $(a_n)$, $(b_n)$ be two nonnegative real sequences such that
$$a_{n+1}\leq a_n - b_n.$$
Then $(a_n)$ is bounded and $\lim_{n\r\infty}b_n = 0$.
\end{lemma}

In order to establish the rate of convergence, we need the following

\begin{lemma}\label{conv_seq}
  Let $(a_n)$, $(b_n)$ be two nonnegative real sequences and $\a_0,\b\in (0,1)$ such that for all $\a\in(0,\a_0)$ the following holds
  \begin{equation}\label{rate_1}
    a_{n+1} + b_{n+1}\leq     (1-2\a)a_n+\a a_{n-1} + \b b_n.
  \end{equation}
  Then there exist $\c \in (0,1)$ and $M>0$ such that $a_{n}\leq \c^n M$ for any $n>0$.
\end{lemma}
\begin{proof}
  First, note that a simple calculus ensures that $f(x) =\frac{1-2x + \sqrt{1+4x^2}}{2} $ is
  decreasing on $(0,1)$. Since $f$ is continuous and $f(0) = 1$, we can choose $\a\in (0,\a_0)$ such
  that $f(\a)\geq \b$. Our next task is to find some $\c\in (0,1)$ and $\d>0$ such that
  \eqref{rate_1} with $\a$ defined as above will be equivalent to the inequality
$$  a_{n+1} + \d a_n + b_{n+1}\leq     \c (a_n+\d a_{n-1}) + \b b_n.  $$
It is easy to check that such numbers are
\begin{align*}
  \c & =  \dfrac{1-2\a + \sqrt{1+4\a^2}}{2}\in (0,1),\\
  \d & = \dfrac{-1+2\a + \sqrt{1+4\a^2}}{2}>0.
\end{align*}
Then by $\c = f(\a)\geq \b$ we can conclude that
\begin{equation*}
  a_{n+1} + \d a_n + b_{n+1}\leq     \c (a_n+\d a_{n-1} + b_n). 
\end{equation*}
Iterating the last inequality simply leads to the desired result
\begin{equation*}
  a_{n+1} + \d a_n + b_{n+1}\leq     \c (a_n+\d a_{n-1} + b_n) \leq \dots \leq \c^n (a_1 + \d a_{0} + b_1) = \c^n M,
\end{equation*}
where $M = a_1 + \d a_0 + b_1 >0 $.
\end{proof}

\section{Algorithm and its convergence}\label{algorithm}
 We first note that solutions of \eqref{vip}
coincide with zeros of the following projected residual function:
$$r(x,y) := \n{y-P_C(x-\la F(y))} + \n{x-y},$$ where $\la$ is some positive number. Now we formally state our algorithm.
\begin{algo} \label{alg}
\begin{enumerate}
  \item Choose $x_0=y_0 \in H$ and
    $\la\in (0,\frac{\sqrt 2 -1}{L})$.
  \item \label{step} Given the current iterate $x_n$ and $y_n$, compute
\begin{equation*}
x_{n+1} = P_C(x_n- \la F(y_n))
\end{equation*}
\item If $ r(x_n,y_n)=0$ then stop: $x_n=y_n=x_{n+1}$ is a   solution. Otherwise   compute 
$$y_{n+1}=2x_{n+1}-x_n,$$
  set $n:=n+1$ and return to step \ref{step}.
\end{enumerate}
\end{algo}
Next lemma is central to our proof of the convergence theorem. 
\begin{lemma}\label{main_lemma}  Let $(x_n)$ and $(y_n)$ be two sequences generated by Algorithm \ref{alg} and let $z\in S$. Then
\begin{align}\label{ineq_lemma}  
  \n{x_{n+1}-z}^2 \leq {}& \n{x_n-z}^2 - (1-\la L(1+\sqrt 2))   \n{x_n-x_{n-1}}^2  +  \la L \n{x_n-y_{n-1}}^2
    \nonumber \\ &- (1-\sqrt 2 \la L) \n{x_{n+1}-y_n}^2  - 2\la \lr{F(z),y_n-z}.
\end{align}
\end{lemma}

\begin{proof}
By Lemma \ref{proj}  we have 
\begin{eqnarray}\label{fst}
  \n{x_{n+1}-z}^2 && \leq \n{x_n-\la F(y_n) -z}^2 - \n{x_n-\la   F(y_n)-x_{n+1}}^2 \nonumber  \\ 
&& = \n{x_n-z}^2 - \n{x_{n+1}-x_{n}}^2 -  2\la \lr{F(y_n),x_{n+1}-z}.
  \end{eqnarray}
Since $F$ is monotone, $2\la \lr{F(y_n)-F(z), y_n- z} \geq 0$. Thus, adding this item to the right side of~\eqref{fst}, we get  
\begin{align}\label{snd} 
  \n{x_{n+1}-z}^2   \leq {} & \n{x_n-z}^2 - \n{x_{n+1}-x_{n}}^2 +  2\la  \lr{F(y_n), y_n-x_{n+1}}  -2\la \lr{F(z),y_n-z} \nonumber\\
  = {} & \n{x_n-z}^2  - \n{x_{n+1}-x_n}^2 +  2\la  \lr{F(y_n)-F(y_{n-1}),y_n - x_{n+1}} \nonumber\\ 
    & +  2\la \lr{F(y_{n-1}),y_n-x_{n+1}}  -2\la \lr{F(z),y_n-z}.
\end{align}
As $x_{n+1},x_{n-1}\in C$, we have by Lemma \ref{proj} 
\begin{align*}
\lr{x_n-x_{n-1} + \la F(y_{n-1}),x_n-x_{n+1}} & \leq 0, \\
\lr{x_n-x_{n-1} + \la F(y_{n-1}),x_n-x_{n-1}} & \leq 0.
\end{align*}
Adding these two inequalities yields
\begin{equation*}
\lr{x_n-x_{n-1}+\la F(y_{n-1}),y_n - x_{n+1}}\leq 0,
\end{equation*}
from which we conclude
\begin{align}\label{est1}
  2\la \lr{F(y_{n-1}), y_n-x_{n+1}}  & \leq 2 \lr{x_n-x_{n-1},x_{n+1}-y_n}   =  2 \lr{y_n-x_n,x_{n+1}-y_n} \nonumber \\ 
& =\n{x_{n+1}-x_n}^2 - \n{x_n-y_n}^2 - \n{x_{n+1}-y_n}^2,
\end{align}
since $x_n-x_{n-1} = y_n - x_n$.

We next turn to estimating $\lr{F(y_n)-F(y_{n-1}),x_{n+1}-y_n}$. It follows that
\begin{multline}\label{est2} 
2\la \lr{F(y_n)-F(y_{n-1}), y_n - x_{n+1}} \leq 2\la L \n{y_n-y_{n-1}}\n{x_{n+1}-y_n} \\ 
\leq \la L\bigl(\frac{1}{\sqrt 2}\n{y_n-y_{n-1}}^2 + \sqrt{2} \n{x_{n+1}-y_n}^2\bigr) \\
\leq \la L \frac{1}{\sqrt 2}\bigl((2+\sqrt 2)\n{y_n-x_n}^2+\sqrt 2\n{x_n-y_{n-1}}^2\bigr) +  \sqrt{2} \la L \n{x_{n+1}-y_n}^2 \\ 
=\la L (1+\sqrt 2)\n{y_n-x_n}^2+ \la L \n{x_n-y_{n-1}}^2 +  \sqrt{2} \la L \n{x_{n+1}-y_n}^2).
\end{multline}

Using \eqref{est1} and \eqref{est2}, we deduce in \eqref{snd} that
\begin{align*}
  \n{x_{n+1}-z}^2 \leq {} &  \n{x_n-z}^2 - (1-\la L(1+\sqrt 2))   \n{x_n-y_n}^2 + \la L \n{x_n-y_{n-1}}^2 \\
   & - (1-\sqrt 2 \la L) \n{x_{n+1}-y_n}^2   - 2\la \lr{F(z),y_n-z}.
\end{align*}
which completes the proof.
\end{proof}

Now we can state and prove our main convergence result.
\begin{theorem}\label{th_1}
Assume that (C1)--(C3) hold. Then any sequence $(x_n)$ generated by Algorithm \ref{alg} weakly converges to a solution  of $\eqref{vip}$.
\end{theorem}
 \begin{proof}
 Let us show that the sequence $(x_n)$ is bounded. Fix any $z\in S$.
Since $ 1-\sqrt 2 \la L > \la L$ and
\begin{align*}
  \lr{F(z),y_n-z} & = 2\lr{F(z),x_n - z} - \lr{F(z),x_{n-1}-z}\\ & \geq
\lr{F(z),x_n - z} - \lr{F(z),x_{n-1}-z},
\end{align*}
we can deduce from inequality~\eqref{ineq_lemma} that
\begin{equation}
\begin{aligned}\label{rewrite_ineq}
    \n{x_{n+1}-z}^2  +{}& \la L \n{x_{n+1}-y_n}^2 +2 \la \lr{F(z),x_n -z }\\
   \leq {}& \n{x_n-z}^2  +\la L \n{x_n-y_{n-1}}^2 + 2\la   \lr{F(z),x_{n-1}-z} \\
    & -  (1-\la L(1+\sqrt 2)) \n{x_n-x_{n-1}}^2.
\end{aligned}
\end{equation}
For $n\geq 2$ let
\begin{align*}
a_n & = \n{x_n-z}^2  + \la L \n{x_n-y_{n-1}}^2 + 2\la   \lr{F(z),x_{n-1}-z},\\
b_n & =  (1-\la L(1+\sqrt 2))\n{x_n-x_{n-1}}^2.
\end{align*}
Then we can rewrite inequality~\eqref{rewrite_ineq} as $a_{n+1}\leq a_n - b_n$. By Lemma~\ref{lim_seq}, we conclude that $(a_n)$ is bounded and $\lim_{n\r \infty}\n{x_{n}-x_{n-1}} = 0$. Therefore, the sequences $(\n{x_n-z})$ and hence $(x_n)$ are also
bounded. From the inequality 
$$\n{x_{n+1}-y_n}\leq \n{x_{n+1}-x_n}+\n{x_n-y_n} = \n{x_{n+1}-x_n}+\n{x_n-x_{n-1}}$$
we also have that $ \lim_{n\r \infty}\n{x_{n+1}-y_{n}} = 0$.
  
As $(x_n)$ is bounded, there exist a subsequence $(x_{n_i})$ of $(x_n)$ such that $(x_{n_i})$ converges weakly to some  $x^*\in H$. It is clear that $(y_{n_i})$ is also convergent to that $x^*$. We  show $x^* \in S$. From Lemma~\ref{proj} it follows that
$$\lr{x_{n_i+1}-x_{n_i} + \la F(y_{n_i}), y-x_{n_i+1}} \geq 0 \quad \forall y\in C.$$
From this we conclude that,  for all $y\in C$,
\begin{align}\label{weak} 
  0 &\leq \lr{x_{n_i+1}-x_{n_i}, y-x_{n_i+1}}+\la \lr{F(y_{n_i}), y-y_{n_i}} +\la \lr{F(y_{n_i}),y_{n_i}-x_{n_i+1}} \nonumber \\  
 & \leq \lr{x_{n_i+1}-x_{n_i}, y-x_{n_i+1}} +\la \lr{F(y), y-y_{n_i}} +\la \lr{F(y_{n_i}),y_{n_i}-x_{n_i+1}}.
\end{align}
In the last inequality we used condition (C2). Taking the limit as $i\r \infty $ in \eqref{weak} and using that $\lim_{i\r\infty}\n{x_{n_i+1}-x_{n_i}} = \lim_{i\r\infty } \n{y_{n_i+1}-y_{n_i}} = 0$, we obtain 
$$0\leq \lr{F(y), y-x^*} \quad \forall y\in C,$$
which implies by Lemma \ref{minty} that $x^* \in S$.

Let us show $x_{n}\rr x^*$. From \eqref{rewrite_ineq} it follows that the sequence $(a_n)$ is monotone for any $z\in S$. Taking into account its boundedness we deduce that it is convergent. At last, the sequence $(\n{x_{n}-y_{n-1}}^2)$ is also convergent, therefore, $(\n{x_n-z}^2+2\la \lr{F(z),x_{n-1}-z})$ is convergent. 

We want to prove that $(x_n)$ is weakly convergent. On the contrary, assume that the sequence $(x_n)$ has at least two weak cluster points $\bar x\in S$ and $\tilde x\in S$ such that $\bar x \neq \tilde x$. Let $(x_{n_k})$ be a sequence such that  $x_{n_k}\rr \bar x$ as $k\to \infty $. Then by Lemma \ref{liminf} and \ref{opial} we have
\begin{multline*}
  \lim_{n\to \infty}\bigl(\n{x_n-\bar x}^2+ 2\la \lr{F(\bar x),x_n-\bar x}\bigr)\\ = \lim_{k \to \infty} \big(\n{x_{n_k}-\bar x}^2+ 2\la \lr{F(\bar x),x_{n_k} - \bar x}\bigr)= \lim_{k\to \infty}\n{x_{n_k}-\bar x}^2 \\= \liminf_{k\to \infty}\n{x_{n_k}-\bar x}^2  < \liminf_{k\to \infty}\n{x_{n_k}-\tilde x}^2 \\ \leq \liminf_{k\to \infty}\n{x_{n_k}-\tilde x}^2 + 2\la  \liminf_{k\to \infty} \lr{F(\tilde x),x_{n_k}-\tilde x} \\\leq  \liminf_{k\to \infty}\bigl(\n{x_{n_k}-\tilde x}^2 + 2\la  \lr{F(\tilde x),x_{n_k}-\tilde x}\bigr) \\ = \lim_{n\to \infty}\bigl(\n{x_{n}-\tilde x}^2 + 2\la  \lr{F(\tilde x),x_{n}-\tilde x}\bigr). 
\end{multline*}
We can now proceed analogously to the proof that
$$
\lim_{n\to \infty}\bigl(\n{x_{n}-\tilde x}^2 + 2\la \lr{F(\tilde x\bigr), x_{n}-\tilde x}\bigr) < \lim_{n\to \infty}\bigl(\n{x_n-\bar x}^2+ 2\la \lr{F(\bar x),x_n-\bar x}\bigr),
$$
 which is impossible. Hence we can conclude that $x_n$ weakly converges to  some $x^*\in S$.
\end{proof}

It is well-known (see \cite{tseng:95}) that under some suitable conditions the extragradient method has R-linear rate of convergence. In the following theorem we  show that our method has the same rate of convergence under a strongly monotonicity assumption of the mapping $F$, i.e.,
$$\text{(C2*)}\qquad \qquad \lr{F(x)-F(y),x-y}\geq m\n{x-y}^2 \quad \forall x,y \in H$$
for some $m>0$.

\begin{theorem}\label{R-linear}
Assume that (C1), (C2*), (C3) hold. Then any sequence $(x_n)$ generated by
Algorithm~\ref{alg} converges to the solution of \eqref{vip} at least R-linearly.  
\end{theorem}
\begin{proof}
Since $F$ is  strongly monotone, \eqref{vip} has a unique solution, which we denote by $z$. Note that by Lemma~\ref{u_v} 
\begin{equation*}
  \n{y_n - z}^2   =2\n{x_n-z}^2-\n{x_{n-1}-z}^2 + 2\n{x_{n}-x_{n-1}}^2\geq 2\n{x_n-z}^2-\n{x_{n-1}-z}^2 
\end{equation*}
From this and from (C2*) we conclude that, for all $m_1\in (0,m]$, 
\begin{multline}\label{str_mon}
2\la\bigl(\lr{F(y_n)-F(z), y_n- z} - m_1(2\n{x_n-z}^2 - \n{x_{n-1}-z}^2)\bigr)  \\
\geq   2\la\bigl(\lr{F(y_n)-F(z), y_n - z} - m\n{y_n-z}^2\bigr)\geq 0.
\end{multline}
From now on, let  $m_1$ be any number in $(0,m]$. Then adding the left part of~\eqref{str_mon} to the right side of~\eqref{fst}, we get
\begin{align*} 
  \n{x_{n+1}-z}^2  \leq {} &  \n{x_n-z}^2 - \n{x_{n+1}-x_{n}}^2 +  2\la \lr{F(y_n), y_n-x_{n+1}} \nonumber  \\
 & - 2\la  \lr{F(z),y_n-z}  - 2\la m_1 (2\n{x_n-z}^2  - \n{x_{n-1}-z}^2)\nonumber \\
 ={} & (1-4\la m_1)\n{x_n-z}^2  -\n{x_{n+1}-x_n}^2 + 2\la m_1\n{x_{n-1}-z}^2  \nonumber \\
& + 2\la \lr{F(y_n)-F(y_{n-1}),y_n - x_{n+1}}    + 2\la \lr{F(y_{n-1}),y_n-x_{n+1}} \nonumber \\ 
& - 2\la \lr{F(z),y_n-z}.
\end{align*}
For items $\lr{F(y_{n-1}),y_n-x_{n+1}}$ and $\lr{F(y_n)-F(y_{n-1}),y_n - x_{n+1}}$ we use  estimations \eqref{est1} and \eqref{est2} from Lemma~\ref{main_lemma}. Therefore, we obtain
\begin{multline*}
  \n{x_{n+1}-z}^2 \leq  (1-4\la   m_1)\n{x_n-z}^2+2\la m_1   \n{x_{n-1} - z}^2 \\ 
- (1-\la L(1+\sqrt 2))   \n{x_n-x_{n-1}}^2  - (1-\sqrt 2   \la L) \n{x_{n+1}-y_n}^2 \\  
+  \la L \n{x_n-y_{n-1}}^2 - 2\la \lr{F(z),y_n-z}.
\end{multline*}
From the last inequality it follows that 
\begin{multline}\label{rate_ineq1}
   \n{x_{n+1}-z}^2 + (1-\sqrt 2 \la L) \n{x_{n+1}-y_n}^2 + 4\la\lr{F(z),x_n-z} \leq  (1-4\la   m_1)\n{x_n-z}^2 \\ 
+ 2\la m_1   \n{x_{n-1} - z}^2 + \la L   \n{x_n-y_{n-1}}^2 + 2\la   \lr{F(z),x_{n-1}-z} \\
\leq (1-4\la  m_1)\n{x_n-z}^2 + 2\la m_1   \n{x_{n-1} - z}^2 \\
+\max\Bigl\{\frac{\la L}{1-\sqrt 2 \la L},\frac 1 2 \Bigr\}\bigl((1-\sqrt 2\la L)\n{x_n -y_{n-1}}^2 + 4\la \lr{F(z),x_{n-1}-z}\bigr).
\end{multline}
In the first inequality we used that $\lr{F(z),y_n-z} = 2\lr{F(z),x_n-z}-\lr{F(z),x_{n-1}-z}$, and in the second we used that 
\begin{align*}
\la L & \leq (1-\sqrt2 \la L) \max\Bigl\{\frac{\la L}{1-\sqrt 2  \la L},\frac 1 2 \Bigr\}, \\
2  & \leq 4\max\Bigl\{\frac{\la L}{1-\sqrt 2  \la L},\frac 1 2 \Bigr\} .
\end{align*}
Set
\begin{align*}
a_{n} & = \n{x_n - z}^2, \\ 
b_n   & =  (1-\sqrt  2 \la L)\n{x_n - y_{n-1}}^2 + 4\la   \lr{F(z),x_{n-1}-z},\\
\b  & = \max\Bigl\{\frac{\la L}{1-\sqrt 2  \la L},\frac 1 2 \Bigr\},\\
\a   & = 2\la m_1\\ 
\a_0 & = 2\la m. 
\end{align*}
As $m_1$ is arbitrary in $(0,m)$, we can rewrite \eqref{rate_ineq1} in the new notation as
\begin{equation*}
  a_{n+1} + b_{n+1} \leq (1-2\a)a_n   + \a a_{n-1}  + \b b_{n} \quad \forall   \a\in (0,\a_0] .
\end{equation*}
Since $\b < 1$, we can conclude by Lemma~\ref{rate_1} that $a_{n}\leq \c^n M$ for some $\c\in(0,1)$ and $M>0$. This means that $(x_n)$ converges to $z$ at least R-linearly. 
\end{proof}
\section{Modified Algorithm}\label{modif_algorithm}
The main shortcoming of all algorithms mentioned in \S\ref{intro}  is a requirement to know the Lipschitz constant or at least to know some estimation of it. Usually it is difficult to  estimate the Lipschitz constant  more or less precisely, thus stepsizes will be quite tiny and, of course, this is not practical. For this reason,  algorithms with constant stepsize are  not applicable in most cases of interest. The usual approaches to  overcome this difficulty  consist in  some prediction of a stepsize with its further correction  (see \cite{khobotov,marcotte:91,tseng00}) or  in a usage of  Armijo-type linesearch  procedure along  a feasible direction (see \cite{iusem:97,solodov:1999}). Usually the latter approach is more effective, since very often the former approach requires too many projections onto the feasible set per iteration.

Nevertheless, our modified method uses the prediction-correction strategy. However, in contrast to algorithms in \cite{khobotov,marcotte:91,tseng00}, we need at most two projections per iteration. This is explained by the fact that for direction $y_n$ in Algorithm~\ref{alg} we use a very simple and cheap formula: $y_n = 2x_n - x_{n-1}$. Although we can not explain this theoretically, but numerical experiments show that cases  with two projections per iteration are quite rare, so usually we have only one projection per iteration that is a drastic contrast to other existing methods.

Looking on the proof of Theorem~\ref{th_1}, we can conclude that  inequality~\eqref{est2} is the only
 place where we use Lipschitz constant $L$. Therefore, choosing $\la_n$  such that the inequality $\la_n \n{F(y_n)-F(y_{n-1})} \leq \a
 \n{y_n-y_{n-1}}$ holds for some  fixed $\a < 1$ we can obtain the  similar estimation as in  \eqref{est2}.
All this leads us to the following 
\begin{algo}\label{algModif0}
\begin{equation*}
\begin{cases}  
\text{Choose}\quad  x_0=y_0\in H, \la_0>0, \a\in (0, \sqrt 2 -1).\\
\text{Choose}\quad \la_n \quad \text{s.t.}\quad  \la_n \n{F(y_n)-F(y_{n-1})} \leq \a \n{y_n-y_{n-1}}, n>0.\\
x_{n+1} = P_C(x_n-\la_n F(y_n)),\\
y_{n+1} = 2x_{n+1} - x_n.
\end{cases}
\end{equation*}
\end{algo}

Although numerical results showed us effectiveness of Algorithm~\ref{algModif0}, we can not prove its convergence for all
cases. Nevertheless, we want to notice that we did not find any example where  Algorithm~\ref{algModif0} did not work. Thus, even Algorithm~\ref{algModif0} seems to be very reliable for many problems.

Now our task is to  modify Algorithm \ref{algModif0} in a such way that we will be able to prove
convergence of the obtained algorithm. For this we need to distinguish  the good cases, where
Algorithm \ref{algModif0} works well, from the bad ones, where it possibly does not.

From now on we adopt the convention that $0/ 0 = +\infty$. Clearly, it follows that
$1/0 = +\infty$ as well. The following algorithm gets round
the difficulty of bad cases of  Algorithm~\ref{algModif0}.
\begin{algo}\label{algModif1}
\begin{remunerate}
  \item Choose $x_0 \in C$, $\la_{-1}>0$,    $\t_0 = 1$, $\a\in (0, \sqrt 2 - 1)$ and
    some large $\bar \la>0$. Compute
    \begin{align*}
      y_0 & = P_C(x_0 - \la_{-1} F(x_0)),\\
      \la_0 & = \min\Bigl\{\frac{\a \n{x_0-y_{0}}}{\n{F(x_0) - F(y_0)}}, \bar \la \Bigr\},\\
      x_1 & = P_C (x_0 - \la_0 F(y_0)).
     \end{align*}
  \item \label{modifstep} Given $x_n$ and $x_{n-1}$, set $\t_n = 1 $ and let $\la(y,\t)$ be defined by
\begin{equation}  \label{la_def}
\la(y,\t):= \min\Bigl\{\frac{\a \n{y-y_{n-1}}}{\n{F(y) -  F(y_{n-1})}},\frac{1+\t_{n-1}}{\t} \la_{n-1}, \bar \la \Bigr\}.
 \end{equation}
 Compute
\begin{align*}
    y_n & = 2x_n - x_{n-1}, \\
\la_n & = \la(y_n,\t_n),\\
x_{n+1} &  = P_C(x_n-\la_n F(y_n)).
\end{align*}
\item If $ r(x_n,y_n)=0$ then stop: $x_n=y_n=x_{n+1}$ is a
  solution. Otherwise compute
\begin{align*}
  t_n = {} &- \n{x_{n+1}-x_n}^2 + 2\la_n\lr{F(y_n),y_n-x_{n+1}}+(1-\a (1+\sqrt 2))\n{x_n-y_n}^2\nonumber \\ 
  & - \a \n{x_n-y_{n-1}}^2 + (1-\sqrt{2}\a)\n{x_{n+1}-y_n}^2.
\end{align*}
\item \label{step_la_n} If $t_n \leq 0$ then set   $n:=n+1$ and go to  step~\ref{modifstep}. Otherwise
  we have two cases $\la_n\geq \la_{n-1}$ and $\la_{n} < \la_{n-1}$.
\begin{romannum}
\item  If $\la_n \geq \la_{n-1}$  then choose $\la_n'\in [\la_{n-1},\la_n]$ such that 
\begin{equation}\label{la_n'}
\n{\la_n' F(y_n) - \la_{n-1}F(y_{n-1})}\leq \a \n{y_n-y_{n-1}}.
\end{equation}
Compute $$x_{n+1} = P_C(x_n - \la_n' F(y_n)).$$
Set $\la_{n}:=\la_n'$, $n:=n+1$ and go to step~\ref{modifstep}.
\item
If $\la_n<\la_{n-1}$ then find $\t_n'\in (0,1]$  such that 
\begin{align} \label{la_nF}
&y_{n}' = x_n + \t_n' (x_n-x_{n-1}) \nonumber \\
& \la(y_n',\t_n') \geq \t_n' \la_{n-1}.
\end{align}
Then choose
$\la_n'\in [\t_n' \la_{n-1}, \la(y_n',\t_n')]$ such that
\begin{equation}\label{la_nD}
\n{\la_n' F(y_n') - \t_n' \la_{n-1}  F(y_{n-1})}\leq \a \n{y_n'-y_{n-1}}.
\end{equation}

Compute
\begin{align*}
x_{n+1} & = P_C(x_n -\la_n' F(y_n'))
\end{align*}
Set $\la_n:=\la_{n}'$, $\t_n:=\t_n'$, $y_n:=y_n'$, $n:=n+1$ and go to step~\ref{modifstep}.
\end{romannum}
\end{remunerate}
\end{algo}
It is clear that  on every iteration in Algorithm~\ref{algModif1} we need to use a residual function with different $\la$, namely,  $r(x_n,y_n) = \n{y_n-P_C(x_n-\la_n F(y_n))} + \n{x_n-y_n}$.

First, let us show that Algorithm~\ref{algModif1} is correct, i.e., it is always possible to choose   $\la_n'$ and $\t_n'$ on steps (\ref{step_la_n}.i) and (\ref{step_la_n}.ii). For this we need two simple lemmas.

\begin{lemma}\label{4i}
  Step (\ref{step_la_n}.i) in Algorithm~\ref{algModif1} is well-defined.
\end{lemma}

\begin{proof}
 From the inequality 
$$\n{\la_{n-1} F(y_n) - \la_{n-1}F(y_{n-1})}\leq \la_n\n{F(y_n)-F(y_{n-1})}\leq \a \n{y_n-y_{n-1}}$$
we can see that it is sufficient to take 
$\la_{n}'=\la_{n-1}$. (However, it seems better for practical reasons to choose $\la_n'$ as great as possible).
\end{proof}

\begin{lemma}\label{4ii}
  Step (\ref{step_la_n}.ii) in Algorithm~\ref{algModif1} is well-defined.
\end{lemma}
\begin{proof}
  First, let us show that $\la(y,\t) \geq \a/L$ for all $y, y_n\in H$ and $\t\in (0,1]$. It is clear that for every $u,v\in H$
$$\frac{\a\n{u-v}}{\n{F(u)-F(v)}}\geq \frac{\a\n{u-v}}{L\n{u-v}} = \frac{\a}{L}.$$
Then $\la_0\geq \a/L$ and hence by induction  
$$ \la(y,\t):= \min\Bigl\{\frac{\a \n{y-y_{n-1}}}{\n{F(y) - F(y_{n-1})}}, \frac{1+\t_{n-1}}{\t} \la_{n-1}, \bar \la \Bigr\}\geq \frac \a L.$$ 
Therefore, it is sufficient to take $\t_n' =\dfrac{\a}{\la_{n-1}L}\in (0,1]$. (But, as above, it seems better to choose $\t_n'$ as great as possible).

At last, we can prove the existence of $\la_n' \in [\t_n' \la_{n-1}, \la(y_n',\t_n')]$ such that \eqref{la_nD} will hold by the same arguments as in  Lemma~\ref{4i}.
\end{proof}

The following lemma  yields an analogous inequality to \eqref{ineq_lemma}.
\begin{lemma}\label{main_modif_lemma}
  Let $(x_n)$ and $(y_n)$ be two sequences generated by Algorithm~\ref{algModif1} and let $z\in S$, $\a \in (0,\sqrt 2 -1)$. Then
\begin{align}\label{modif_ineq}
  \n{x_{n+1}-z}^2 \leq {} &  \n{x_n-z}^2 - (1-\a (1+\sqrt 2))\n{x_n-y_n}^2  - (1-\sqrt 2\a) \n{x_{n+1}-y_n}^2\nonumber \\ 
  &+\a  \n{x_n-y_{n-1}}^2 - 2\la_n \lr{F(z),y_n-z}.
\end{align}
\end{lemma}
\begin{proof}
Proceeding analogously as in \eqref{fst} and in \eqref{snd}, we get
\begin{align}\label{modif_fst}
  \n{x_{n+1}-z}^2  \leq {} &  \n{x_n-z}^2  - \n{x_{n+1}-x_n}^2 +  2\la_n  \lr{F(y_n)-F(y_{n-1}),y_n - x_{n+1}} \nonumber\\ 
 &   + 2\la_n \lr{F(y_{n-1}),y_n-x_{n+1}}  -2\la_n \lr{F(z),y_n-z}.
\end{align}
The same arguments as in \eqref{est1} yield
\begin{align}\label{modif_est1}
  2\la_{n-1}\lr{F(y_{n-1}),y_n-x_{n+1}} & \leq 2 \lr{x_n-x_{n-1},x_{n+1}-y_n} = 2 \lr{y_n-x_n,x_{n+1}-y_n} \nonumber \\ 
& =\n{x_{n+1}-x_n}^2 - \n{x_n-y_n}^2 - \n{x_{n+1}-y_n}^2.
\end{align} 
Using  \eqref{est2} and the inequality $\la_n \n{F(y_n) - F(y_{n-1})}\leq \a \n{y_n-y_{n-1}}$, 
we obtain 
\begin{multline}\label{modif_est2}
2\la_n \lr{F(y_n)-F(y_{n-1}), y_n-x_{n+1}} \leq 2\a\n{y_n-y_{n-1}}\n{x_{n+1}-y_n} \\ 
\leq \a (1+\sqrt 2)\n{y_n-x_n}^2 + \a  \n{x_n-y_{n-1}}^2 + \sqrt{2} \a \n{x_{n+1}-y_n}^2.
\end{multline}
Applying \eqref{modif_est1} and \eqref{modif_est2} in \eqref{modif_fst}, we  conclude that
\begin{align}\label{whatnext?}
  \n{x_{n+1}-z}^2 \leq {}& \n{x_n-z}^2 - \n{x_{n+1}-x_n}^2\nonumber \\
   & + \a\bigl((1+\sqrt 2)\n{x_n-y_n}^2  + \n{x_n-y_{n-1}}^2 +  \sqrt{2} \n{x_{n+1}-y_n}^2 \bigr)\nonumber  \\
 & + \frac{\la_{n}}{\la_{n-1}} \bigl(\n{x_{n+1}-x_n}^2 - \n{x_n-y_n}^2 - \n{x_{n+1} - y_n}^2\bigr) \nonumber \\ 
 & - 2\la_n \lr{F(z),y_n-z}.
\end{align}
Now we have a lot of cases
\begin{remunerate}
\item[(a)] $\la_n = \la_{n-1}$; 
\item[(b)] $\la_n<\la_{n-1}$ and $\n{x_{n+1}-x_n}^2 >
 \n{x_n-y_n}^2 - \n{x_{n+1} - y_n}^2$;
\item[(c)] $\la_n>\la_{n-1}$ and $\n{x_{n+1}-x_n}^2\geq
 \n{x_n-y_n}^2 - \n{x_{n+1} - y_n}^2$;
\item[(d)] $\la_n<\la_{n-1}$ and $\n{x_{n+1}-x_n}^2\leq
 \n{x_n-y_n}^2 - \n{x_{n+1} - y_n}^2$;
\item[(e)] $\la_n>\la_{n-1}$ and $\n{x_{n+1}-x_n}^2 <
 \n{x_n-y_n}^2 - \n{x_{n+1} - y_n}^2$.
\end{remunerate}
Let us consider all them.

Case (a). It is quite obvious that when $\la_n =\la_{n-1}$ we obtain inequality~\eqref{modif_ineq}.

Cases (b) and (e). In both cases we notice that
$$ \bigl(1-   \frac{\la_n}{\la_{n-1}}\bigr) \n{x_{n+1}-x_n}^2 > \bigl(1-  \frac{\la_n}{\la_{n-1}}\bigr)\bigl (\n{x_n-y_n}^2+\n{x_{n+1}-y_n}^2\bigr).$$
Thus, from \eqref{whatnext?} we conclude that
\begin{align*}
  \n{x_{n+1}-z}^2 \leq {} & \n{x_n-z}^2 -\bigl(1- \frac{\la_n}{\la_{n-1}}\bigr) \n{x_{n+1}-x_n}^2 \\
   & + \a\bigl((1+\sqrt 2)\n{x_n-y_n}^2  + \n{x_n-y_{n-1}}^2 +  \sqrt{2} \n{x_{n+1}-y_n}^2 \bigr)  \\
& -\frac{\la_{n}}{\la_{n-1}} \bigl( \n{x_n-y_n}^2 + \n{x_{n+1} - y_n}^2\bigr) -  2\la_n \lr{F(z),y_n-z} \\
 < {} & \n{x_{n}-z}^2 - \n{x_{n}-y_{n}}^2 - \n{x_{n+1}-y_n}^2 - 2\la_n \lr{F(z),y_n-z}\\
& + \a\bigl((1+\sqrt 2)\n{x_n-y_n}^2  + \n{x_n-y_{n-1}}^2 +  \sqrt{2} \n{x_{n+1}-y_n}^2 \bigr)\\
 = {}& \n{x_n-z}^2 - (1-\a (1+\sqrt 2))\n{x_n-y_n}^2   - (1-\sqrt 2\a) \n{x_{n+1}-y_n}^2 \\ 
& + \a  \n{x_n-y_{n-1}}^2 - 2\la_n \lr{F(z),y_n-z},
\end{align*}
which is exactly inequality \eqref{modif_ineq}.

Observe that inequality~\eqref{whatnext?} with condition  $t_n\leq 0$ is equivalent to~\eqref{modif_ineq}. So, from now on we assume that $t_n>0$.

Case (c) and $t_n>0$. It is clear that we can rewrite \eqref{modif_fst} (with $\la_n'$ instead of $\la_n$) as
\begin{align*}\label{modif_snd}
  \n{x_{n+1}-z}^2  \leq {}&  \n{x_n-z}^2  - \n{x_{n+1}-x_n}^2 +  2  \lr{\la_n' F(y_n)- \la_{n-1} F(y_{n-1}),y_n - x_{n+1}} \\ 
&   + 2\la_{n-1} \lr{F(y_{n-1}),y_n-x_{n+1}}  -2\la_n' \lr{F(z),y_n-z} \\
\leq {} & \n{x_n-z}^2  - \n{x_{n+1}-x_n}^2 +  2 \a \n{y_n-y_{n-1}}\n{y_n-x_{n+1}}  \\  &   + 2\la_{n-1} \lr{F(y_{n-1}),y_n-x_{n+1}}  -2\la_n' \lr{F(z),y_n-z}.
\end{align*}
Now using \eqref{modif_est1} and the second inequality in \eqref{modif_est2},  we get easily~\eqref{modif_ineq}.

Case (d) and $t_n>0$. 
Let $\la_{n-1}' = \t_n' \la_{n-1}$ and $x_{n-1}'  = x_n - \t_n' (x_n-x_{n-1})$.  Since $\la_{n-1}'\leq \la_{n-1}$, the point $x_{n-1}'$ lies on the segment $[x_{n-1},x_n]$ and the point $x_{n-1}' - \la_{n-1}'F(y_{n-1})$ lies on the segment $[x_n, x_{n-1}-\la_{n-1}
F(y_{n-1})]$. Thus by Lemma~\ref{sunny}
\begin{equation*}
   x_{n}  =  P_C (x_{n-1}' - \la_{n-1}' F(y_{n-1})).
\end{equation*}
Also we have that  $x_{n+1} = P_C (x_{n} - \la_n' F(y_n'))$  and $y_n' = 2x_n - x_{n-1}'$. Applying the same arguments as above for this new $x_{n+1}$, we do not get into the Case (d), since $\la_n'\geq \la_{n-1}'$, and do not get into the Case (c), since 
$$\n{\la_n' F(y_n') - \la_{n-1}' F(y_{n-1})}\leq \a \n{y_n'-y_{n-1}}.$$
This means that we get
\begin{align*}
  \n{x_{n+1}-z}^2 \leq {}&  \n{x_n-z}^2 - (1-\a (1+\sqrt 2))\n{x_n-y_n'}^2   - (1-\sqrt 2\a) \n{x_{n+1}-y_n'}^2\\
 & +\a  \n{x_n-y_{n-1}}^2 - 2\la_n' \lr{F(z),y_n'-z}.
\end{align*}
Hence, for redefined values we obtain the desired inequality~\eqref{modif_ineq} and the proof is complete.
\end{proof}
 
Before proceeding, we want to give some intuition to Algorithm~\ref{algModif1}. First, note that $\bar \la$ is the upper bound of
$(\la_n)$. It guarantees that our  stepsize does not go to  infinity. The item $\frac{1+\t_{n-1}}{\t}\la_{n-1}$ in \eqref{la_def} provides us necessary condition to obtain  the good estimation of $2\la_n\lr{F(z),y_n-z}$. This seems not very important, since in many cases $\lr{F(z),y_n-z}\geq 0$ and we can just not take into account this item in~\eqref{modif_ineq}.   But this condition also  ensures that  $\la_n$ is not too large in comparison with  $\la_{n-1}$. This may be useful if we get into the Case (c) or (d) in Lemma~\ref{main_modif_lemma}, since we will be able  to perform \eqref{la_n'} or \eqref{la_nD} faster. 
Also note that  usually we have $\la_n \leq 2\la_{n-1}$, since $\t_n\neq 1$ only in the Case (d).

The inequality $t_n\leq 0$  ensures at once  that on the $n$-th iteration  inequality~\eqref{modif_ineq} holds. Therefore, we do not need to immerse into a branch of the cases in Algorithm~\ref{algModif1} or Lemma~\ref{main_modif_lemma}. Also note that a former and a latter use different classifications of possible cases. In Lemma~\ref{main_modif_lemma} the inequality $t_n\leq 0$ holds certainly for cases (a), (b), and (e). But it happens very often that even in such bad cases as (c) or (d) $t_n\leq 0$ also. Thus, by checking whether  $t_n\leq 0$, we can avoid to get into the undesirable steps~(\ref{step_la_n}.i) or~(\ref{step_la_n}.ii). 

If $t_n>0$ and $\la_n\geq \la_{n-1}$, then we get into step~ (\ref{step_la_n}.i). Actually, this case is not so bad. We have two possibilities here. The first one is to decrease $\la_n$ until it satisfies \eqref{la_n'}. And the second one is just to set $\la_{n}'=\la_{n-1}$. The first one seems to be better, since we suppose, as usual, that a larger stepsize provides a smaller number of iterations (however, it needs some extra work to check condition \eqref{la_n'}).

It only remains the most tedious case when $t_n>0$ and $\la_n< \la_{n-1}$. In this case we show that
it is always possible to change $y_n=2x_n-x_{n-1}$ to some other $y_n$, defined by $y_n =
(1+\t_n)x_n - \t_n x_{n-1}$ with some $\t_n\in (0,1)$, which  we choose to ensure that~\eqref{la_nF}
and~\eqref{la_nD} hold. Next, with new $y_n$ and $\la_n$ we  compute $x_{n+1}$ as in
step~\ref{modifstep}. The conditions~\eqref{la_nF} and~\eqref{la_nD} guarantee that for new $t_n$ we
have $t_n\leq 0$. However, when we go to the next iteration, we must not forget that $y_n$ and
$\la_n$ have been changed.

Now we prove the convergence of Algorithm~\ref{algModif1} in a quite similar way as in \S~\ref{algorithm}.
\begin{theorem} 
Assume that (C1)--(C3) hold. Then  any sequences $(x_n)$ and $(y_n)$ generated by Algorithm \ref{algModif1} weakly converge to a solution  of \eqref{vip}.
\end{theorem}
 \begin{proof}
 Let us show that the sequence $(x_n)$ is bounded. Fix any $z\in S$.
 Since $\t_n \la_n\leq (1+\t_{n-1})\la_{n-1}$, we have that
\begin{align*}
\la_n\lr{F(z),y_n-z} & = \la_n\lr{F(z),(1+\t_n)x_n - \t_{n} x_{n-1} - z} \\ 
& = \la_n(1+\t_n)\lr{F(z),x_n - z} - \la_n\t_n \lr{F(z),x_{n-1}-z} \\ 
& \geq \la_n(1+\t_n)\lr{F(z),x_n - z} - \la_{n-1}(1+\t_{n-1})\lr{F(z),x_{n-1} - z}.
\end{align*}
And hence we can deduce from inequality \eqref{modif_ineq} that 
\begin{align}\label{modif_re1}
    \n{x_{n+1}-z}^2  +{}& \a \n{x_{n+1}-y_n}^2 + 2\la_n(1+\t_n)\lr{F(z),x_n - z} \nonumber \\
 \leq{}&  \n{x_n-z}^2 + \a  \n{x_n-y_{n-1}}^2  + 2\la_{n-1}(1+\t_{n-1})\lr{F(z),x_{n-1} - z}\nonumber  \\
 & - (1-\a (1+\sqrt 2))(\n{x_n-y_n}^2 + \n{x_{n+1}-y_n}^2).
 \end{align}

For $n\geq 1$  set
\begin{align*}
a_{n+1} & = \n{x_{n+1}-z}^2  + \a \n{x_{n+1}-y_n}^2 + 2\la_n(1+\t_n)\lr{F(z),x_n - z},\\
b_n & = (1-\a (1+\sqrt 2))(\n{x_n-y_n}^2 + \n{x_{n+1}-y_n}^2).
\end{align*}
With this notation we can rewrite inequality~\eqref{modif_re1} as
$a_{n+1}\leq a_n - b_n$. Thus by Lemma \ref{lim_seq} we can conclude that $(a_n)$, and hence $(x_n)$, are bounded and $\lim_{n\r\infty}b_n = 0$.
From the latter it follows immediately that
\begin{equation*}
\lim_{n\r\infty}\n{x_{n+1}-y_n}=0\quad \text{and}\quad \lim_{n\r\infty}\n{x_{n+1}-x_n}=0.
\end{equation*}

From this point we need to proceed in much the same way as in Theorem~\ref{th_1}. Thus, to avoid the repetition, we restrict our attention to only the place where arguments differ. This is  inequality~\eqref{weak} where we want to go to the limit as $i\r\infty$. Recall that $\bar \la$  and $\frac \a L$ are the upper and the lower bounds of the sequence $(\la_n)$, respectively (the former follows form the definition of $\la_n$ and the latter from Lemma~\ref{4ii}). Consequently, we can conclude that the sequence $(\la_n)$ has a cluster point, and moreover, all its cluster points are not equal to zero. This means  that when we go to the limit in \eqref{weak} we obtain the same result as in Theorem~\ref{th_1}. This is the desired conclusion.
\end{proof}

Algorithm~\ref{algModif1} allows several modifications. In particular, instead of checking whether $t_n\leq 0$, one can require at once that $\la_n$, which we choose in step~\ref{modifstep},  satisfies the following inequality
 $$\n{\la_n F(y_n) - \la_{n-1}F(y_{n-1})}\leq \a \n{y_n-y_{n-1}}.$$
The arguments, used in  steps~(\ref{step_la_n}.i) and~(\ref{step_la_n}.ii), show that is always possible to do (perhaps, with changed $y_n$). Although this algorithm  have only one projection per iteration, our computational experience shows that it is less effective than Algorithm~\ref{algModif1}. This is explained again by the fact that very often $t_n\leq 0$ even for not suitable $\la_n$.

\section{Numerical illustration}\label{num_results}
In this section we examine the benchmark of Alg.~\ref{alg} and~\ref{algModif1}. For this, first, we
present the performance of stationary algorithms on some affine problems. And then we compare
nonstationary Alg.~\ref{algModif1} with two well-known methods. For abbreviation    we use the following.

Stationary methods:
\begin{itemize}
    \item  EGM for algorithm described by \eqref{korpel};

    \item SubEGM for algorithm described by \eqref{censor};

    \item TBFM for algorithm described by \eqref{tseng};
    \item SubPM for algorithm described by \eqref{mal_sem}.
\end{itemize}

Nonstationary methods:
\begin{itemize}
    \item Alg. Solod.-Svait. for the  algorithm of Solodov and Svaiter \cite{solodov:1999};

    \item TBFM-nst for the nonstationary backward-forward algorithm of Tseng \cite{tseng00}.

\end{itemize}

Note that in \cite{solodov:1999} there were considered two algorithms, so we chose the second one with the better performance.

   For a benchmark of all algorithms we included the number of iterations (\texttt{iter.}) and the time of execution (\texttt{time}). For the nonstationary algorithms we also added the total number (\texttt{np}) of projections and the total number (\texttt{nf}) of all values $F$ that were evaluated.

We describe the test details below. Computations were performed using Wolfram Mathematica 8.0 on an Intel Core i3-2348M \@2.3GHz running 64-bit Linux Mint 13.  The time is measured in seconds using the intrinsic Mathematica function \texttt{Timing}. The projection onto the feasible set $C$ was performed using the intrinsic Mathematica function \texttt{ArgMin} (for Problems 1, 4, and 5 the projection is explicit). The termination criteria are the following
\begin{align}
  &\text{EGM, SubEGM, TBFM, Solod-Svait., TBFM-nst.}  &&\n{x_n-y_n} \leq \e ,\label{1_cr}\\
  &\text{SubPM} &&r(x_{n+1},y_n) \leq \e, \label{2_cr}\\
  &\text{Alg.~\ref{alg}, \ref{algModif1}} &&r(x_n,y_n) \leq \e. \label{3_cr}
\end{align}
We do not use the simple criterion $\n{x-P_C(x-\la F(x))}\leq \varepsilon$ for Alg.~\ref{alg},
\ref{algModif1}, and SubPM because we do not want to compute $P_C(x-\la F(x))$ extra.  The better way is to verify that
\begin{align*}
  \n{y_n- P_C(y_n-\la F(y_n))}  &\leq \n{x_{n+1}-y_n}+\n{x_{n+1}-P_C(y_n-\la F(y_n))}\\
  & \leq \n{x_{n+1}-y_n}+\n{x_n-y_n} = r(x_n,y_n).
\end{align*}
The last inequality holds, since $P_C$ is nonexpansive. The similar argument works for
SubPM. Moreover, we notice that termination criteria \eqref{2_cr}, \eqref{3_cr} are
stronger than \eqref{1_cr}. This means that the solutions generated by Alg.~\ref{alg},
\ref{algModif1}, and SubPM are not worse (or even better) than the solutions generated by other algorithms if we measure the quality of the solution by the projected error bound.
\begin{table}\centering
  \footnotesize
  \begin{tabular}{||c||c|c||c|c||c|c||c|c||c|c|c||} 
    \hline\hline
    $ m$ & \multicolumn{2}{c||}{EGM} &
    \multicolumn{2}{c||}{SubEGM } &
    \multicolumn{2}{c||}{SubPM} &
    \multicolumn{2}{c||}{Alg. \ref{alg}} \\ \hline
    &  iter. & time &  iter. & time &  iter. & time &  iter. &   time  \\ \hline
    500   & 129 & 1.3 & 129  & 1.3 & 109  & 0.6 & 92  & 0.5\\ \hline
    1000  & 133 & 5.6 & 133  & 5.7 & 120  & 2.6 & 95  & 2\\ \hline
    2000  & 138 & 22  & 138  & 22  & 121  & 10  & 98  & 8\\ \hline
    4000  & 143 & 98  & 143  & 101 & 122  & 43  & 101 & 33 \\   \hline\hline
  \end{tabular}\caption{Results for EGM, SubEGM, SubPM, and Alg~\ref{alg} for Problem~1. $L=1$, $\la = 0.4$, $\e = 10^{-3}$}\label{table1}
\end{table}

\begin{remunerate}
\item[\textsf{Problem 1.}] \label{pr1} Our first example is classical. The feasible set is $C = \R^{m}$ and $F(x) = A x$, where $A$ is a square matrix $m\times m$ defined by condition
$$
a_{ij}=
\begin{cases}
  -1, &\quad  j=m+1-i>i \\
  1, & \quad j=m+1-i<i \\
  0, & \quad \text{otherwise.}
\end{cases}
$$
This corresponds to a matrix whose secondary diagonal consists half of $-1$’s and half of $1$’s, the rest of the elements being zero. For even $m$ the zero vector is the problem solution of \eqref{vip}. This is a classical example of a problem where usual gradient method does not converge. The results of the numerical experiments are compiled in Table~\ref{table1}.  We did not include the results of TBFM, since for this problem it is identical with EGM. For all tests we took $x^0 = (1,\dots, 1)$, $\e = 10^{-3}$ and $\la = 0.4$.

\item[\textsf{Problem 2.}]\label{pr2} We take $F(x) = Mx + q$ with the matrix $M$ randomly generated as suggested in \cite{hphard}:
$$M = AA^T + B + D,$$ where every entry of the $m \times m$ matrix $A$ and of the $m\times m$ skew-symmetric matrix $B$ is uniformly generated from $(-5, 5)$, and every diagonal entry of the $m\times m$ diagonal $D$ is uniformly generated from $(0, 0.3)$ (so $M$ is positive definite), with every entry of $q$ uniformly generated from $(-500, 0)$. The feasible set is 
$$C = \{x\in \R^m_+ \mid x_1 + x_2 +\dots + x_m = m\},$$ and the starting point is $x_0 = (1,\dots, 1)$.  For this problem we took $L = \n{M}$. For every $m$, as shown in Table~\ref{table2}, we have generated three random samples with different choice of $M$ and $q$. The error is $\e = 10^{-3}$.
\begin{table}\centering
  \footnotesize
  \begin{tabular}{||c||c|c||c|c||c|c||c|c||c|c||c|c||c||}
    \hline\hline
    $m$ & \multicolumn{2}{c||}{EGM} &
    \multicolumn{2}{c||}{SubEGM } &
    \multicolumn{2}{c||}{TBFM} &  \multicolumn{2}{c||}{SubPM} &
    \multicolumn{2}{c||}{Alg. \ref{alg}} \\ \hline
    &  iter. & time &  iter. & time & iter. & time &  iter. & time &  iter. & time  \\ \hline
    10 & 82  & 7.7  & 82  & 3.9 & 82  & 3.9 & 83  & 4.0 & 83  &4.0  \\ \hline
    10 & 146 & 15.2 & 148 & 7.8 & 148 & 7.8 & 152 & 8.0 & 147 & 7.8 \\ \hline
    10 & 63  & 6.3  & 63  & 3.2 & 63  & 3.2 & 63  & 3.2 & 63  & 3.2 \\ \hline
    20 & 78  & 10.2 & 79  & 5.2 & 79  & 5.3 & 79  & 5.3 & 79  & 5.3 \\ \hline
    20 & 101 & 14.5 & 102 & 7.3 & 102 & 7.4 & 101 & 7.3 & 102 & 7.4 \\ \hline
    20 & 110 & 17.6 & 110 & 8.9 & 110 & 8.9 & 110 & 8.9 & 111 & 9.0 \\ \hline
    50 & 224 & 243  & 226 & 126 & 226 & 128 & 231 & 134 & 225 & 134 \\ \hline
    50 & 236 & 277  & 237 & 144 & 238 & 143 & 241 & 151 & 237 & 144 \\ \hline
    50 & 444 & 507  & 447 & 265 & 447 & 266 & 455 & 263 & 445 & 267\\ 
    \hline\hline
  \end{tabular}\caption{Results for EGM, SubEGM, TBFM, SubPM, and  Alg~\ref{alg}
    for Problem~2. $L=\n{M}$,  \qquad \qquad  $\la =  0.4/L$, $\e = 10^{-3}$}\label{table2}
\end{table}

As one can see from Table~\ref{table2}, EGM has the worst performance.  This is explained by the
presence of two projections per iteration, which played the key role, since for small $m$ the
computation of $F$ is cheap. All others algorithms behave similarly. Also it is very likely that
with an increasing of $m$ they will continue to have approximately the same number of
iterations. However, for large $m$ the computation of $F$ will also matter, so for those cases we
can expect that SubPM and Alg.~\ref{alg} will have substantially better performance, compared with the others. \\ \par 

From now on, we present out test results for nonstationary methods.  
The choice of test problems for nonstationary algorithms is explained by a simple structure of the feasible set $C$. This is important for the better performance of TBFM-nst, since one must compute many projections per iteration, and Alg.~Solod.-Svait., since on every iteration one must compute a projection onto the intersection of $C$ with some hyperspace. (Of course, instead of a projection onto this intersection, we can project onto the set $C$. But in that case the stepsize will be smaller).

In the implementation of Alg.~Solod-Svait. we chose $\sigma = 0.3$, $\c=0.5$, $\theta = 4$; in the
implementation of TBFM-nst. we chose $\b=0.5$, $\theta = 0.9$ where the notation is taken from
\cite{solodov:1999} and \cite{tseng00} respectively. For Alg.~\ref{algModif1} we used $\a = 0.4$ and
$\la_{-1} = 0.01$.
\item[\textsf{Problem 3.}] Kojima-Shindo Nonlinear Complementarity Problem (NCP) was considered in~\cite{pang:93} and in~\cite{kanzow:94}, where $m=4$ and the mapping $F$ is defined by
$$F(x_1,x_2,x_3,x_4) =\begin{bmatrix}
  3x_1^2 + 2x_1x_2 + 2x_2^2 + x_3 +
  3x_4 -6 \\
  2x_1^2 + x_1+x_2^2+10x_3+2x_4-2\\
  3x_1^2 + x_1x_2 + 2x_2^2+2x_3 + 9x_4
  -9\\
  x_1^2+3x_2^2+2x_3+3x_4-3
\end{bmatrix}.
$$ 
The feasible set is $C = \{x\in \R^4_+ \mid x_1 + x_2 + x_3 + x_4= 4\}$.  First, we chose some particular starting points: $x^0 = (1,1,1,1)$ and $x^0=(0.5,0.5,2,1)$. And then we used different random points from $C$. For all starting points we had two tests: with $\e=10^{-3}$ and $\e= 10^{-6}$. The results are summarized in Table~\ref{table3}. Here we aborted the evaluation of algorithm after $50$ seconds.

As one can see,  the ratio $nf/iter.$ for Alg. Solod.-Svait. is much higher than for Alg.~4.2. For
this reason even with larger  number of iterations Alg.~4.2 sometimes outperforms Alg. Solod.-Svait. Note that TBFM-nst has the worst performance, since on every iteration it must compute many nontrivial projections onto the simplex (this is exactly what we meant at the beginning of \S\ref{modif_algorithm}). The solution of this problem is not unique, but most of the times all algorithms converge to the $x^* = (1.225,0, 0, 2.775)$. But for example for $x^0=(1,1,1,1)$ algorithms converge to the different solutions, this is the reason of such distinction in their performance. Also for $\e = 10^{-6}$ Alg.~Solod.-Svait. often converges very slowly, that is why it does not execute in $50$ seconds. In that time, Alg.~\ref{algModif1} and TBFM-nst have not this drawback, since their stepsizes $\la_n $ are separated from zero.

\begin{table}\centering
  \footnotesize
  \begin{tabular}{||c|c|c|c|c|c|c|c||}
    \hline\hline
    $x^0$ & $\e$ & \multicolumn{2}{c|}{Alg. \ref{algModif1}} &
    \multicolumn{2}{c|}{Alg. Solod.-Svait.} & \multicolumn{2}{c||}{ TBFM-nst} \\ \hline
    & & iter. (np/nf)  & time &  iter. (np/nf) & time & iter. (np/nf) & time \\ \hline 
    $(1,1,1,1)$ & $10^{-3}$ & 36 (36/36)   & 1.0 & 6 (11/11)  & 0.4  & 62 (66/128) & 2.1 \\ \hline 
    & $10^{-6}$ & 72 (82/86)  & 2.2 & 6 (11/11) & 0.4  &156 (160/316) & 4.9 \\ \hline 

    $(0.5,0.5, 2,1)$ & $10^{-3}$ & 41 (41/41) &  1.1 & 18 (35/58) & 1.4 &69 (73/142) & 2.2 \\ \hline
    & $10^{-6}$ & 75 (87/86) & 2.4  & --- & $>$ 50 & 163 (167/330) & 5.0 \\ \hline
    
    random & $10^{-3}$ & 57 (57/58) & 1.6 & 13 (25/49) & 1.2 &72 (77/149) & 2.3\\ \hline
    & $10^{-6}$ & 63 (67/65) & 1.9 & --- & $>$ 50 &79 (84/163) & 2.6 \\ \hline

    random &$10^{-3}$ & 40 (40/40) & 1.0 & 20 (39/73) & 1.5  & 66 (70/136) &2.0\\ \hline
    &$10^{-6}$ & 71 (82/74)& 2.0 & --- &$>$ 50 & 160 (164/324)  & 4.9 \\ \hline

    random &$10^{-3}$ & 59 (59/59) & 1.5 & 15 (29/40) & 1.1 & 87 (91/178)& 2.6 \\ \hline
    &$10^{-6}$ & 93 (99/105)  & 2.5 & --- & $>$ 50 & 181 (185/366) &5.7 \\ \hline\hline
  \end{tabular}\caption{Results for Alg.~\ref{algModif1}, Alg.~Solod.-Svait., and TBFM-nst for Problem 3}  \label{table3}
\end{table}

\item[\textsf{Problem 4.}] This example was considered by Sun in~\cite{sun:94}, where
  \begin{align*}
    F(x)   & = F_1(x) + F_2(x),\\
    F_1(x) & = (f_1(x),f_2(x),\dots,   f_m(x)),\\
    F_2(x) & = Dx+c, \\
    f_i(x) & = x_{i-1}^2 + x_i^2 +  x_{i-1}x_i + x_i x_{i+1},\quad   i=1,2,\dots, m\\
    x_0 & = x_{m+1} = 0,
  \end{align*}
  Here $D$ is a square matrix $m\times m$ defined by condition
$$d_{ij}=
\begin{cases}
  4, & i = j,\\
  1, & i-j=1,\\
  -2,& i-j = -1,\\
  0, & \text{otherwise},
\end{cases}
$$
and $c=(-1,-1,\dots, -1)$.  The feasible set is $C = \R^m_+$ and the starting point is $x^0 = (0,0,\dots,0)$. The test results are listed in Table~\ref{table4} for different $m$ and $\e$.
\begin{table}\centering
  \footnotesize
  \begin{tabular}{||c|c|c|c|c|c|c|c||}
    \hline\hline
    $x^0$ & $\e$ & \multicolumn{2}{c|}{Alg. \ref{algModif1}} &
    \multicolumn{2}{c|}{Alg. Solod.-Svait.} & \multicolumn{2}{c||}{ TBFM-nst} \\ \hline
    & & iter. (np/nf)  & time &  iter. (np/nf) & time & iter. (np/nf) & time \\ \hline 
    5 & $10^{-3}$ & 20 (20/20) & 0.008 & 18 (35/70) & 0.012 & 22 (24/46) & 0.007 \\ \hline
    & $10^{-6}$ & 43 (43/43) & 0.014 &  34 (67/134) & 0.018 &47 (49/96) & 0.019 \\  \hline

    50 & $10^{-3}$ &23 (24/26) & 0.02 & 20 (39/78) & 0.06 & 28 (30/58) & 0.04 \\ \hline
    & $10^{-6}$ &46 (47/49) & 0.04 & 34 (67/134) & 0.1 & 53 (55/108) & 0.07 \\\hline 

    500 &$10^{-3}$ & 27 (28/30) & 0.6 & 22 (43/86) & 1.6 & 32 (34/66) & 1.0 \\ \hline
    &$10^{-6}$ & 50 (51/53) & 1.4 & 36 (71/142) & 2.7 & 58 (60/118) & 1.8 \\ \hline

    1000 &$10^{-3}$ & 28 (29/31) & 1.4 & 22 (43/86) & 4.8 & 33 (35/68) & 2.8 \\ \hline
    &$10^{-6}$ & 51 (52/54) & 2.6 & 36 (71/142)  &7.4 & 59 (61/120) & 5.1 \\ 
    \hline\hline
  \end{tabular}\caption{Results for Alg.~\ref{algModif1} and Alg.~Solod.-Svait., and TBFM-nst 
   for Problem 4}  \label{table4}
\end{table}

\item[\textsf{Problem 5.}] This problem was used by Kanzow in~\cite{kanzow:94} with $m=5$. The mapping $F$ is defined by
  \begin{align*}
    F(x) & = (f_1(x),\dots,  f_5(x)),\\
    f_i(x) & = 2(x_i-i+2)\exp\Bigl\{\sum_{i=1}^5 (x_i-i+2)^2\Bigr\},\quad i=1,\dots, 5н
  \end{align*}
  and the feasible set is $C = \R^5$. The solution for this problem is $x^* = (-1,0,1,2,3)$.  The results of the numerical experiments for different starting points and $\e$ are compiled in Table~\ref{table5}.
\end{remunerate}

We want to note that for this problem only the number of iterations matter, since both projection and mapping $F$ are cheap. For this reason, the time is quite small and hence, can not be a good illustration of the algorithm's performance, since  Mathematica function  \texttt{Timing} measures time only with some error of order $10^{-3}$.
\begin{table}\centering
  \footnotesize
  \begin{tabular}{||c|c|c|c|c|c|c|c||}
    \hline\hline
    $x^0$ & $\e$ & \multicolumn{2}{c|}{Alg. \ref{algModif1}} &
    \multicolumn{2}{c|}{Alg. Solod.-Svait.} & \multicolumn{2}{c||}{TBFM-nst} \\ \hline
    & & iter. (np/nf)  & time &  iter. (np/nf) & time & iter. (np/nf) & time \\ \hline 
    $(1,1,1,1,1)$ & $10^{-3}$ & 26 (26/26)   & 0.009 & 14 (27/53)  & 0.009  & 28 (29/57) & 0.007 \\ \hline 
    & $10^{-6}$ & 49 (49/49)  & 0.017 &  24 (47/93) & 0.015& 52 (53/105) & 0.014 \\ \hline 

    $(0,0,0,0,0)$ &$10^{-3}$ & 15 (18/35) & 0.009 & 32 (63/82) & 0.015  & 43 (44/87)& 0.012 \\ \hline
    &$10^{-6}$ & 34 (37/54) & 0.016 & 42 (83/122) & 0.025 & 67 (68/135) & 0.017 \\ \hline

    random & $10^{-3}$ &25 (28/41) & 0.012 & 79 (157/175) & 0.029 & 90 (91/181) & 0.024  \\ \hline
    & $10^{-6}$ &59 (59/59) & 0.02 & 89 (179/215) & 0.050 & 114 (115/229) &0.036 \\\hline 

    random &$10^{-3}$ & 20 (24/44)  &0.015 & 23 (45/67) & 0.011 & 36 (37/73) &0.017\\ \hline
    &$10^{-6}$ & 43 (47/67) & 0.022 & 33 (65/107) & 0.017 &60 (61/121) &0.016\\ 
    \hline\hline
  \end{tabular}\caption{Results for Alg.~\ref{algModif1} and Alg.~Solod.-Svait., and TBFM-nst
    for Problem 5}  \label{table5}
\end{table}

To summarize our numerical experiments, we want to make some observations. First, the disadvantage
of Alg.~\ref{alg} in comparison with other stationary algorithms is a smaller interval for possible
stepsize $\la$. For simple problems it really can matter, however, for huge-scale problems
evaluation of $F$ is much more important. To illustrate this, let us consider Problem 1 again. The
best performance of EGM and SubEGM achieves approximately for $\la = 0.7$ (SubPM and Alg.~\ref{alg}
do not converge for such $\la$). Nevertheless, the time execution of EGM and SubEGM for $m=2000$ is
$12$ seconds for both, and for $m=4000$ is $52$ and $53$ seconds, respectively. Hence, for Problem~1
with large $m$ EGM and SubEGM, even with their best stepsize, are slower.

Second, our numerical experiments confirm that ``bad'' cases~(\ref{step_la_n}.i) and~(\ref{step_la_n}.ii) in Alg~\ref{algModif1} are really very rare. Notice that the number of evaluation of $F$ is approximately the same as the number of iteration, that is a drastic contrast with Alg.~Solod.Svait and TBFM-nst. Also if even we do not take into account the good performance of Alg.~\ref{algModif1}, it seems to be more robust than Alg.~Solod.-Svait., since it does not require any additional assumptions on the feasible set and its stepsizes are separated from zero. On the other hand, the proof of Alg.~Solod-Svait., in contrast to Alg.~\ref{algModif1},  does not  require the Lipschitz-continuity of the mapping $F$.

\section*{Acknowledgments}
The author wishes to express his gratitude to Mikhail Solodov and to the two anonymous referees
for their constructive suggestions that led to improvements in the paper.
\bibliographystyle{siam}

\bibliography{publications}

\end{document}